\documentclass[10pt]{amsart}

\usepackage{amsmath,amsfonts, latexsym,graphicx, footmisc, amssymb, mathtools, enumerate}

\usepackage{hyperref}
\hypersetup{colorlinks=true, urlcolor=blue, citecolor=blue, linkcolor=blue}

\newcommand{\U}{{\mathcal U}}
\newcommand{\0}{{\mathbf 0}}
\newcommand{\C}{{\mathbb C}}
\newcommand{\Z}{{\mathbb Z}}
\newcommand{\R}{{\mathbb R}}

\newtheorem{defn0}{Definition}[section]
\newtheorem{prop0}[defn0]{Proposition}
\newtheorem{conj0}[defn0]{Conjecture}
\newtheorem{thm0}[defn0]{Theorem}
\newtheorem{lem0}[defn0]{Lemma}
\newtheorem{corollary0}[defn0]{Corollary}
\newtheorem{example0}[defn0]{Example}
\newtheorem{remark0}[defn0]{Remark}
\newtheorem{question0}[defn0]{Question}
\newtheorem{exercise0}[defn0]{Exercise}

\newenvironment{defn}{\begin{defn0}}{\end{defn0}}
\newenvironment{prop}{\begin{prop0}}{\end{prop0}}

\newenvironment{thm}{\begin{thm0}}{\end{thm0}}

\newenvironment{cor}{\begin{corollary0}}{\end{corollary0}}

\newenvironment{exm}{\begin{example0}\rm}{\end{example0}}

\newcommand{\propref}[1]{Proposition~\ref{#1}}
\newcommand{\thmref}[1]{Theorem~\ref{#1}}

\newcommand{\corref}[1]{Corollary~\ref{#1}}

\title{Milnor fibers of homogeneous polynomials of prime power degree}


\author{David B. Massey}

\date{}

\begin{document}

\begin{abstract} We consider a homogeneous polynomial of degree equal to a prime power and examine the lowest ``interesting'' degree cohomology of the Milnor fiber.
 \end{abstract}

\maketitle




\section{Introduction}

Let $\U$ be an open neighborhood of the origin in $\C^{n+1}$ and let $f:(\U,\0)\rightarrow(\C, 0)$ be a nowhere-locally-constant analytic function. We use $(z_0, \dots, z_n)$ for coordinates on $\U$, and let $\Sigma f$ denote the critical locus of $f$, i.e., 
$$
\Sigma f:=V\left(\frac{\partial f}{\partial z_0},\dots, \frac{\partial f}{\partial z_n}\right).
$$
 Note that near the origin, $\Sigma f\subseteq V(f)$. Let $s:=\dim_0\Sigma f$.

Consider the Milnor fiber $F_{f, \0}$ of $f$ at the origin. The reduced cohomology $\widetilde H^*(F_{f,\0};\,\Z)$ is possibly non-zero only in degrees $k$ where $n-s\leq k\leq n$. Thus, we consider degree $(n-s)$ to be the lowest degree cohomology which is ``interesting''. These $\Z$-modules are finitely-generated and $\widetilde H^{n-s}(F_{f,\0};\,\Z)$ is free abelian, but it is in general unknown how to calculate these $\Z$-modules or their ranks.

Suppose that $\mathcal S$ is a Whitney stratification of $V(f)$ in a neighborhood of the origin. For each $s$-dimensional irreducible component $C$ of $\Sigma f$ through the origin, let $C^\circ$ denote the unique stratum of $\mathcal S$ such that $C=\overline{C^\circ}$.  For each $p\in C^\circ$ near $\0$, there is $F_{f, p}$, the Milnor fiber of $f$ at $p$. If $N$ is a normal slice to $C$ at $p$, $f_{|_N}$ has an isolated critical point at $p$, and $F_{f, p}$ is homeomorphic to the product of a polydisk with $F_{f_{|_N}, p}$. Thus, $\widetilde H^{*}(F_{f,p};\,\Z)=0$ except in degree $(n-s)$ and
$$
\widetilde H^{n-s}(F_{f,p};\,\Z)\cong \widetilde H^{n-s}(F_{f_{|_N}, p};\,\Z)\cong\Z^{\mu^\circ_C},
$$
where $\mu^\circ_C$ is the Milnor number of $f_{|_N}$ at $p$; this is independent of both $p$ and the normal slice $N$. We refer to $\mu_C^\circ$ as the {\bf generic Milnor number of $f$ on $C$}.

In terms of the derived category and perverse sheaves, this cohomology $\widetilde H^{n-s}(F_{f,p};\,\Z)$ for $p\in C^\circ$ is the stalk local system obtained by restricting the shifted complex of vanishing cycles along $f$ to the stratum $C^\circ$. As we shall discuss in Section 2, this local system is characterized by a monodromy representation (the {\bf internal} or {\bf vertical} monodromy) of $\pi_1(C^\circ)$ into $\operatorname{Aut}_\Z\big(\widetilde H^{n-s}(F_{f,p};\,\Z)\big)$. We denote the invariant $\Z$-submodule of this monodromy representation by $\operatorname{inv}\left\{\widetilde{\mathbf F}_{f,C^\circ}\right\}$.

By a result of Siersma (\cite{siersmavarlad}) in the case $s=1$ and, more generally,  by a result of Maxim, P\u aunescu, and Tib\u ar (\cite{MPT}; recalled in \thmref{thm:supersiersma}) using the cosupport condition for the perverse sheaf of shifted vanishing cycles, the cohomology at the Milnor fiber at the origin is not independent of the cohomology of the nearby Milnor fibers; in fact, $\widetilde H^{n-s}(F_{f,\0};\,\Z)$ injects into the direct sum of the $\operatorname{inv}\left\{\widetilde{\mathbf F}_{f,C^\circ}\right\}$, where the direct sum is over all  $s$-dimensional irreducible components $C$ of $\Sigma f$ through the origin. From this, one immediately obtains the inequality that $\operatorname{rank}\widetilde H^{n-s}(F_{f, \0};\Z)\leq \sum_C \mu^\circ_C$, where the summation is over the $s$-dimensional irreducible components $C$ of  $\Sigma f$, and any upper bounds on any $\operatorname{rank}\big(\operatorname{inv}\big\{\widetilde{\mathbf F}_{f,C^\circ}\big\}\big)$ gives a better bound on $\operatorname{rank}\widetilde H^{n-s}(F_{f, \0};\Z)$.

\smallskip

In \thmref{thm:relmono}, we recall an argument of Steenbrink from 6.4 of \cite{steenbrink} which tells us a strong relationship exists between the Milnor monodromy  and the internal monodromy; as our context is different from that of Steenbrink, we supply the proof, but it remains the same. By combining Steenbrink's result with that of Maxim, P\u aunescu, and Tib\u ar, and using the Monodromy Theorem (\cite{clemens}, \cite{sgavii1}, \cite{landman}, \cite{lemono1}) and A'Campo's Theorem (\cite{acamp}) together with a little number theory, we obtain in  \thmref{thm:mainintro}:

\begin{thm} Suppose that $f\in \C[z_0, \dots, z_n]$ is homogeneous of degree $\mathfrak p^m$, where $\mathfrak p$ is prime. Let  $s:=\dim\Sigma f$. 

Then, we have the following upper-bound on the rank of the free abelian group $\widetilde H^{n-s}(F_{f, \0};\Z)$:
$$
\operatorname{rank} \widetilde H^{n-s}(F_{f, \0};\Z)\leq \sum_C \left\lfloor\frac{\mu_C^\circ+\mathfrak p\left\lfloor\frac{\mu^\circ_C-(-1)^{n+1-s}}{\mathfrak p}\right\rfloor +(-1)^{n+1-s}}{2}\right\rfloor,
$$
where the summation is over the $s$-dimensional irreducible components $C$ of $\Sigma f$ and $\lfloor\,\rfloor$ denotes the floor function.
\end{thm}

\medskip

After we prove the above theorem, we provide examples and then make some closing remarks.

\medskip

We wish to thank Laurentiu Maxim for many comments, and in particular for providing a reference to \cite{MPT} for the vanishing cycle result in \thmref{thm:supersiersma}. We also thank Laurentiu for pointing out Steenbrink's proof in 6.4 of \cite{steenbrink} of  the result which is referred to in this paper as \thmref{thm:relmono}. We mistakenly believed that we had proved this first in  \cite{pervdual}; crediting J. Steenbrink for this result is long overdue.

\medskip

\section{Local Systems and the result of Maxim, P\u aunescu, and Tib\u ar}

The point of this section to introduce terminology and notation required to state and understand the statement of the vanishing cycle result in  Theorem 3.4.a of Maxim, P\u aunescu, and Tib\u ar from \cite{MPT}. However, while discussing monodromy, we also recall the statements of the Monodromy Theorem and A'Campo's Theorem on the Lefschetz number of the Milnor monodromy.

\smallskip

Throughout this section, we continue with $\U$ being an open neighborhood of the origin in $\C^{n+1}$,  $f:(\U,\0)\rightarrow(\C, 0)$ being a nowhere-locally constant analytic function, and we are allowing  $\dim_\0\Sigma f$  to be arbitrary.

For each $p\in V(f)$,  for all $k\geq 0$, we have the Milnor monodromy automorphism $T^{(k)}_{f, p}$ acting on $H^k(F_{f, p}; \Z)$; this is obtained by letting the value of $f$ travel counterclockwise once around the base circle of the Milnor fibration and taking the induced automorphism. We have put the degree in parentheses on the monodromy to distinguish it from raising to a power (which we will do later).

For each $p\in V(f)$, the cohomology $H^k(F_{f,p};\,\Z)$ is a finitely-generated $\Z$-module and, hence, is a direct sum of $\Z^{b_k}$ for some non-negative integer $b_k$ (the degree $k$ Betti number) with a torsion submodule. The monodromy $T^{(k)}_{f, p}$is an automorphism of this entire $\Z$-module.

However, for the remainder of this paper, whenever we discuss the determinant, trace, eigenvalues, characteristic polynomials, or Lefschetz numbers of automorphisms of $H^k(F_{f,p};\,\Z)$, we are referring to the automorphisms induced on the free abelian part of $H^{k}(F_{f, p}; \Z)$, that is, induced on $H^{k}(F_{f, p}; \Z)$ modulo its torsion submodule. These automorphisms are represented by invertible integral matrices. Thus the  determinant would be a unit in $\Z$, i.e., $\pm1$ and the trace must be an integer; in fact, the entire characteristic polynomial would have integer coefficients.

\medskip

There is the well-known:

\begin{thm} \textnormal{(The Monodromy Theorem, \cite{clemens}, \cite{sgavii1}, \cite{landman}, \cite{lemono1}):} For all $p\in V(f)$, for each $k$, the Milnor monodromy $T^{(k)}_{f, p}$ is quasi-unipotent, i.e., has complex eigenvalues which are all roots of unity. 
\end{thm}

\medskip

As an easy, but not so well known, corollary to this, we have:

\begin{cor}\label{cor:inversechar} The characteristic polynomial of $\left(T^{(k)}_{f, p}\right)^{-1}$ is equal to the characteristic polynomial of $T^{(k)}_{f, p}$.
\end{cor}
\begin{proof} The characteristic polynomial of $T^{(k)}_{f, p}$ has integer (in particular, real) coefficients. Thus the roots other than $\pm1$, the eigenvalues with multiplicities, occur in complex conjugate pairs. But the conjugate of a root of unity is its reciprocal. Thus the eigenvalues $\neq\pm1$ occur in reciprocal pairs. However the eigenvalues (with multiplicities) of $\left(T^{(k)}_{f, p}\right)^{-1}$ are the reciprocals of the eigenvalues of $T^{(k)}_{f, p}$, and we are finished.
\end{proof}

\medskip

We also have the following well-known theorem:

\begin{thm}\label{thm:acampo}\textnormal{(A'Campo's Theorem,  \cite{acamp}):} Suppose that $p\in\Sigma f\cap V(f)$. Then the Lefschetz number $\mathfrak L\big\{T^{(*)}_{f, p}\big\}$ is equal to $0$.
\end{thm}

\bigskip 

Before stating the result of Maxim, P\u aunescu, and Tib\u ar that we need, we must have a discussion about local systems.

\medskip

\begin{defn}\label{defn:inv} Let $Y$ be a connected complex analytic space. A {\bf local system} $\mathbf L$ of $\Z$-modules  on $Y$ is a locally constant sheaf of $\Z$-modules on $Y$. Specifying such a local system is equivalent to selecting a point $p\in Y$ and specifying a group homomorphism (the monodromy representation)
$$
\rho:\pi_1(Y, p)\rightarrow\operatorname{Aut}_\Z(\mathbf L_p),
$$
where $\mathbf L_p$ is the stalk of $\mathbf L$ at $p$. The global sections of $\mathbf L$ are
$$
H^0(Y; \mathbf L)\cong \bigcap_{\gamma\in \pi_1(Y, p)}\operatorname{ker}\{\operatorname{id}-\rho(\gamma)\};
$$
this is the submodule of $\mathbf L_p$ of elements which are invariant under all of the monodromy actions and, as the isomorphism-type of this submodule is independent of $p$, we denote this submodule by $\operatorname{inv}\{\mathbf L\}$.
\end{defn}

\medskip

Now we need to define a special local system related to the Milnor fiber.

\smallskip

\begin{defn}\label{defn:milnorsys} Let $\U$ be an open neighborhood of the origin in $\C^{n+1}$ and $f:(\U,\0)\rightarrow(\C, 0)$ being a nowhere-locally constant analytic function. Let $s:=\dim_\0\Sigma f$. 

Then, for each irreducible component $C$, there is an analytic Zariski open dense subset $C^\circ\subseteq C$ (e.g., a Whitney stratum in some Whitney stratification of $V(f)$) such that the reduced cohomology $\widetilde H^{n-s}(F_{f, p};\,\Z)$, for $p\in C^\circ$, is locally constant on $C^\circ$ and determines a local system given by automophisms of $\widetilde H^{n-s}(F_{f, p};\,\Z)$ as the point $p$ travels around loops in $C^\circ$.

We call this {\bf the reduced Milnor local system on $C^\circ$} and denote it by $\widetilde{\mathbf F}_{f,C^\circ}$. (In formal perverse sheaf terms, this is the local system $\left(\mathbf H^{-s}(\phi_f[-1]\Z_U^\bullet[n+1])\right)_{|_{C^\circ}}$, the restriction of the shifted vanishing cycles along $f$.) The stalk of this local system is $\widetilde H^{n-s}(F_{f, p};\,\Z)\cong \Z^{\mu_C^\circ}$.
\end{defn}

\medskip

Now, we give the result from Theorem 3.4.a of \cite{MPT}, which is obtained by using the cosupport condition satisfied by the perverse sheaf of shifted vanishing cycles along $f$.

\begin{thm}\label{thm:supersiersma} \textnormal{ (Maxim, P\u aunescu, Tib\u ar) } Let $\U$ be an open neighborhood of the origin in $\C^{n+1}$ and $f:(\U,\0)\rightarrow(\C, 0)$ being a nowhere-locally constant analytic function. Let $s:=\dim_\0\Sigma f$.

Then there is an injection 
$$\widetilde H^{n-s}(F_{f, \0};\Z)\hookrightarrow\bigoplus_{C}\operatorname{inv}\left\{\widetilde{\mathbf F}_{f,C^\circ}\right\},$$ where $C$ runs over the $s$-dimensional irreducible components of $\Sigma f$ at the origin. This injection commutes with the respective Milnor monodromies. 

In particular, $\operatorname{rank}\widetilde H^{n-s}(F_{f, \0};\Z)\leq \sum_C \mu^\circ_C$.

\end{thm}

\medskip

The remainder of this paper is about improving the bound 
$$\operatorname{rank}\widetilde H^{n-s}(F_{f, \0};\Z)\leq \sum_C \mu^\circ_C$$
 from \thmref{thm:supersiersma}, that is, we want to show, for at least one $s$-dimensional component $C$ of $\Sigma f$,  that 
 $$
 \operatorname{rank}\operatorname{inv}\left\{\widetilde{\mathbf F}_{f,C^\circ}\right\}<\mu^\circ_C.
 $$
 Thus, we want to put a ``good'' bound on the geometric multiplicity of the eigenvalue $1$ for at least one monodromy action of the local system $\widetilde{\mathbf F}_{f,C^\circ}$; we accomplish this by putting bounds on the algebraic multiplicity of the eigenvalue $1$.

\medskip

\section{Homogeneous polynomials of arbitrary degree}

\smallskip

In this section, we assume that $f$ is a homogeneous polynomial  in $\C[z_0, \dots, z_n]$ of arbitrary degree $d$  such that $s:=\dim\Sigma f$ is also arbitrary. 

\medskip

We wish first to dispose of an easy, but not completely trivial, case of the more general results we will derive later.

\begin{prop}\label{prop:special}
Suppose that $f\in\C[z_0, \dots, z_n]$ is homogeneous of degree $d$ and $\dim\Sigma f=n$, so that $f$ is not reduced. Let
$$
f=\prod_{k=1}^{r} g_k^{a_k}
$$
be the irreducible factorization of $f$. Let $\mu^\circ_k$ denote the generic Milnor number of $f$ on $V(g_k)$.

Then, 
\begin{enumerate}
\item for each $k$, $\mu_k^\circ=a_k-1$, and the number of connected components of $F_{f, \0}$ is $\operatorname{gcd}\{a_k\}_k$;
\smallskip
\item  $\operatorname{rank} \widetilde H^0(F_{f, \0};\,\Z)=\operatorname{gcd}\{a_k\}_k-1\leq \sum_k(a_k-1)= \sum_k\mu_k^\circ$;
\smallskip

\item Let $S$ be a set of natural numbers such that, for all $m\in S$, $\operatorname{gcd}(m, d)=1$. Let $S-1:=\{m-1\,|\,m\in S\}$. Then, 
$$\operatorname{rank} \widetilde H^0(F_{f, \0};\,\Z)\leq \sum_{\mu_k^\circ\not\in S-1}\mu_k^\circ.$$
\end{enumerate}
\end{prop}
\begin{proof} That $\mu_k^\circ=a_k-1$ is trivial. That the number of connected components of $F_{f, \0}$ is $\operatorname{gcd}\{a_k\}_k$ follows at once from Dimca's Proposition 2.3 of \cite{dimcasing}. Item 2 is immediate from Item 1.

Item 3 requires a small argument. If, for all $k$, $\mu_k^\circ\not\in S-1$, then Item 3 is true by Item 2. Now suppose that there exists $\mu_{k_0}^\circ=a_{k_0}-1\in S-1$, so that $a_{k_0}=m$, where $m\in S$. Then, as $\operatorname{gcd}(m, d)=1$ and $\operatorname{gcd}\{a_k\}_k$ divides $d$, we must have $\operatorname{gcd}\{a_k\}_k=1$. Thus, by Item 2, $\operatorname{rank} \widetilde H^0(F_{f, \0};\,\Z)=0$ and the inequality in Item 3 holds (where, as usual, a summation over an empty indexing set is zero).
\end{proof}

\medskip

Now suppose that $\nu$ is an arbitrary complex line in $V(f)$ which contains the origin. Then $\nu$ is homeomorphic to a disk and so $\nu-\{\0\}$ has the homotopy-type of a circle, with a default ``counterclockwise'' orientation determined by the complex structure on $\nu-\{\0\}$. Let $p\in\nu-\{\0\}$ near $\0$. Then there is the internal monodromy automorphism $Q^{(k)}_{f, \nu}$ which acts on $\widetilde H^{k}(F_{f, p}; \Z)$; this automorphism is induced by leaving the value of $f$ fixed, but letting $p$ travel once counterclockwise around the punctured disk $\nu-\{\0\}$.

\begin{defn} Let $p\in\nu-\{\0\}$. Then, in each degree $k$, there is the corresponding {\bf internal monodromy} automorphism $Q^{(k)}_{f, \nu}$ which acts on $H^{k}(F_{f, p}; \Z)$; this automorphism is induced by leaving the value of $f$ fixed, but letting $p$ travel once counterclockwise around the punctured disk $\nu-\{\0\}$.
\end{defn}

\medskip

We have the following fundamental relationship between our two monodromies. In fact,  in \cite{pervdual}, we derived this relation in the case of central hyperplane arrangements in $\C^3$ by blowing up the origin. However, the result is more generally true for arbitrary homogeneous polynomials, with essentially the same proof, and there was never any need to blow up. While we are in a different setting, the proof that we give is that of Steenbrink in 6.2 of \cite{steenbrink}; we give the proof because it is short and to make it clear that the result holds in our setting.

\begin{thm}\label{thm:relmono} \textnormal{ (Steenbrink) } Suppose $f$ is homogeneous of degree $d$. Let $\nu\subseteq V(f)$ be a complex line through the origin and let $p\in\nu-\{\0\}$. 
Then 
$$
Q^{(k)}_{f,\nu}=\Big(T^{(k)}_{f, p}\Big)^{-d}.
$$
\end{thm}
\begin{proof} We assume without loss of generality that $\nu$ is the $z_0$-axis, and do what one would do if blowing up the origin, without actually blowing up the origin. 

\smallskip

 As in the introduction, if  $N$ is a normal slice to $\nu$ at $p\in\nu-\{\0\}$, $F_{f, p}$ is homeomorphic to the product of a disk with $F_{f_{|_N}, p}$ (but $f_{|_N}$ need no longer have an isolated critical point at $p$). Thus,
$$
H^{*}(F_{f,p};\,\Z)\cong  H^{*}(F_{f_{|_N}, p};\,\Z)
$$
and so both the Milnor monodromy and the internal monodromy can be viewed as acting on $H^{*}(F_{f_{|_N}, p};\,\Z)$.

As coordinates on $\C^{n+1}-V(z_0)$, we use $u_0:=z_0$ and, for $1\leq j\leq n$,  $u_j:=z_j/z_0$. Thus, on $\C^{n+1}-V(z_0)$, our polynomial $f(z_0, \dots, z_n)$ becomes $u_0^d\,f(1, u_1, u_2, \dots, u_n)$.

We select a value $t\in \C^*$; without loss of generality, we fix $u_0=z_0=t=1$. Let $r, \epsilon\in\R$ be such that $0<r\ll\epsilon\ll1$ so that the Milnor fiber of $f_1$ at $\0$ is 
$$
B^\circ_\epsilon\cap f_1^{-1}(r)=\{(u_1,\dots, u_n)\,\big|\, |u_1|^2+\cdots+|u_n|^2<\epsilon^2, \, (1)^df(1, u_1,\dots, u_n)=r\},
$$
where $B_\epsilon^\circ$ now denotes an open ball of radius $\epsilon$ centered at the origin in $\C^n\cong \{1\}\times \C^n$, where we define a ball using the coordinates $u_1, \dots, u_n$ (which is also an open ball of radius $\epsilon$ centered at the origin in $\C^n$ using the coordinates $z_1, \dots, z_n$).

The Milnor monodromy is induced by following the fiber
$$
B^\circ_\epsilon\cap f_1^{-1}(r)=\{(u_1,\dots, u_n)\,\big|\, |u_1|^2+\cdots+|u_n|^2<\epsilon^2, \, (1)^df(1, u_1,\dots, u_n)=re^{i\theta}\}
$$
as $\theta$ goes from $0$ to $2\pi$.

The internal monodromy is induced by following the fiber
$$
B^\circ_\epsilon\cap f_{e^{i\omega}}^{-1}(r)=\{(u_1,\dots, u_n)\,\big|\, |u_1|^2+\cdots+|u_n|^2<\epsilon^2, \, (e^{i\omega})^df(1, u_1,\dots, u_n)=r\}
$$
as $\omega$ goes from $0$ to $2\pi$. The equality in the set can be rewritten as
$$
f(1, u_1,\dots, u_n)=re^{-id\omega}
$$
and the result follows.
\end{proof}

\smallskip

From the Monodromy Theorem and \thmref{thm:relmono},  we immediately conclude that:

\begin{cor}\label{cor:unity} Suppose $f$ is homogeneous. Let $\nu\subseteq V(f)$ be a complex line through the origin. Then all of the eigenvalues of $Q^{(k)}_{f,\nu}$ are roots of unity.
\end{cor}

\medskip

Now \thmref{thm:supersiersma} tells us that
$$\operatorname{rank}\widetilde H^{n-s}(F_{f, \0};\Z)\leq \sum_{C}\operatorname{rank}\left\{\operatorname{inv}\left\{\widetilde{\mathbf F}_{f,C^\circ}\right\}\right\},$$
where, for all $C$, $\operatorname{rank}\left\{\operatorname{inv}\left\{\widetilde{\mathbf F}_{f,C^\circ}\right\}\right\}\leq\mu^\circ_C$. If $C$ is an $s$-dimensional component of $\Sigma f$ and we have a line $\nu\subseteq C^\circ$, then 
$$\operatorname{inv}\left\{\widetilde{\mathbf F}_{f,C^\circ}\right\}\subseteq\ker\big\{\operatorname{id} - Q^{(n-s)}_{f,\nu}\big\}
$$
and so any non-trivial upper bound on the rank of the eigenspace of $1$ for $Q^{(n-s)}_{f,\nu}$ will give us an improved bound on $\operatorname{rank}\widetilde H^{n-s}(F_{f, \0};\Z)$.

\medskip

In the next corollary, we obtain such bounds in special cases. We should remark that in the $n-s=0$ case, Item 3 of \propref{prop:special} gives us the inequalities on $\operatorname{rank}\widetilde H^{0}(F_{f, \0};\Z)$ that we conclude from the corollary below when $n-s\geq 1$.

\medskip

\begin{cor}\label{cor:special} Suppose that $f\in \C[z_0, \dots, z_n]$ is homogeneous of degree $d$ and let $s$ denote $\dim \Sigma f$. Suppose that $n-s\geq 1$ and $s\geq 1$.  Let $C$ be an $s$-dimensional component of $\Sigma f$.

\smallskip

Then,
\begin{enumerate}
\item
\smallskip 
Suppose that $\mu^\circ_C=1$, $d$ is odd, and $n-s$ is even. Then $\operatorname{rank}\left\{\operatorname{inv}\left\{\widetilde{\mathbf F}_{f,C^\circ}\right\}\right\}=0$.
\item \smallskip
Suppose that $\mu^\circ_C=2$, $3\not| d$, and $n-s$ is even. Then $\operatorname{rank}\left\{\operatorname{inv}\left\{\widetilde{\mathbf F}_{f,C^\circ}\right\}\right\}=0$.
\item \smallskip Suppose that $\mu^\circ_C=2$, $6\not| d$, and $n-s$ is odd.  Then $\operatorname{rank}\left\{\operatorname{inv}\left\{\widetilde{\mathbf F}_{f,C^\circ}\right\}\right\}=0$.
\item \smallskip Suppose that $\mu^\circ_C=3$, $d$ is odd, and $n-s$ is even.  Then $\operatorname{rank}\left\{\operatorname{inv}\left\{\widetilde{\mathbf F}_{f,C^\circ}\right\}\right\}\leq 1$.
\item \smallskip Suppose that $\mu^\circ_C=3$, $d$ is odd, and $n-s$ is odd.  Then $\operatorname{rank}\left\{\operatorname{inv}\left\{\widetilde{\mathbf F}_{f,C^\circ}\right\}\right\}\leq 2$.

\end{enumerate}
\end{cor}

\begin{proof} Let $\nu$ be a complex line in $C^\circ$, and let $p\in\nu-\{\0\}$. Of course we will use that
$$\operatorname{inv}\left\{\widetilde{\mathbf F}_{f,C^\circ}\right\}\subseteq\ker\big\{\operatorname{id} - Q^{(n-s)}_{f,\nu}\big\}.
$$

 As $n-s\geq 1$, $\widetilde H^{n-s}(F_{f, p};\,\Z)=H^{n-s}(F_{f, p};\,\Z)$. Recall also that the eigenvalues  (with multiplicities) of $T^{(n-s)}_{f, p}$ are roots of unity which -- for eigenvalues other than $\pm1$ -- must occur in complex conjugate pairs. As a final note before we look at the cases, observe that A'Campo's theorem tells us that the sum of the eigenvalues (with multiplicities) of $T^{(n-s)}_{f, p}$ is $-1$ if $(n-s)$ is even, and is $1$ if $(n-s)$ is odd.

\medskip

\noindent Item 1:  Suppose that $d$ is odd, $n-s$ is even, and $\mu^\circ_C=1$.  Then, since $\mu^\circ_\nu=1$ and $(n-s)$ is even, $T^{(n-s)}_{f, p}$ must be multiplication by $-1$. Now, since $d$ is odd,  \thmref{thm:relmono} tells us that $Q^{(n-s)}_{f, \nu}$ must also be multiplication by $-1$. Thus, $\operatorname{ker}\big\{\operatorname{id}-Q^{(n-s)}_{f, \nu}\big\}=0$ and we are finished.

\bigskip

\noindent Item 2:  Suppose that $n-s$ is even and $\mu^\circ_C=2$.  Then the two eigenvalues of $T^{(n-s)}_{f, p}$  must be
$$
-\frac{1}{2}+\frac{\sqrt{3}}{2}\,i=e^{2\pi i/3}\hskip 0.2in\textnormal{and} \hskip 0.2in-\frac{1}{2}-\frac{\sqrt{3}}{2}\,i=e^{-2\pi i/3}.
$$
Now, since $3\not| d$ and $Q^{(n-s)}_{f,\nu}=\Big(T^{(n-s)}_{f, p}\Big)^{-d}$, we conclude that $1$ is not an eigenvalue of $Q^{(n-s)}_{f,\nu}$, i.e., $\operatorname{ker}\{\operatorname{id} - Q^{(n-s)}_{f,\nu}\}=0$, and we are finished.

\bigskip

\noindent  Item 3:  Suppose that $n-s$ is odd and $\mu^\circ_C=2$. Then the two eigenvalues of $T^{(n-s)}_{f, p}$ must be
$$
\frac{1}{2}+\frac{\sqrt{3}}{2}\,i=e^{\pi i/3}\hskip 0.2in\textnormal{and}\hskip 0.2in \frac{1}{2}-\frac{\sqrt{3}}{2}\,i=e^{-\pi i/3}.
$$
Since $6\not| d$ and $Q^{(n-s)}_{f,\nu}=\Big(T^{(n-s)}_{f, p}\Big)^{-d}$, we conclude that $1$ is not an eigenvalue of $Q^{(n-s)}_{f,\nu}$, i.e., $\operatorname{ker}\{\operatorname{id} - Q^{(n-s)}_{f,\nu}\}=0$, and we are finished.

\bigskip

\noindent Items 4 and 5: If one of the eigenvalues of $T^{(n-s)}_{f, p}$ is not $\pm1$, then its complex conjugate must also be an eigenvalue; this would contradict that the eigenvalues are roots of unity and that sum of the eigenvalues is $\pm1$. Thus, if $(n-s)$ is even, the eigenvalues with multiplicities must be $-1$, $-1$,  $1$ while, if $(n-s)$ is odd, the eigenvalues must be $-1$, $1$, $1$. The stated conclusions follow.
\end{proof}

\medskip

\begin{exm}\label{exm:good} Let $f(x,y,r,t)=r^2y^2-tx^3$. So, $n=3$ and $d=4$. We find
$$
\Sigma f=V(tx^3, r^2y, ry^2, x^4) =V(x,y)\cup V(x, r) =: C_1\cup C_2.
$$
Thus $s=2$ and $n-s=1$ is odd.

\smallskip

To calculate $\mu_{C_1}^\circ$, we let $f_{r_0, t_0}=r_0^2y^2-t_0x^3$, and calculate $\mu_\0(f_{r_0, t_0})$ for generic values of $r_0$ and $t_0$. We find $\mu_{C_1}^\circ=2$.

\smallskip

To calculate $\mu_{C_2}^\circ$, we let $f_{y_0, t_0}=r^2y_0^2-t_0x^3$, and calculate $\mu_\0(f_{y_0, t_0})$ for generic values of $y_0$ and $t_0$. We find $\mu_{C_2}^\circ=2$.

\smallskip

Therefore, both $C_1$ and $C_2$ are the case covered by Item 3 of \corref{cor:special}. Combining this with \thmref{thm:supersiersma} tells us that $H^1(F_{f, \0};\, \Z)=0$.

\end{exm}

\medskip

\section{Homogeneous polynomials of prime power degree}

We continue to assume that $f$ is a homogeneous polynomial in $\C[z_0, \dots, z_n]$ of arbitrary degree $d$, and $s:=\dim \Sigma f$. Let $p\in\Sigma f$.

\smallskip

\thmref{thm:relmono} tells us that, for all $k$, $Q^{(k)}_{f,\nu}=\Big(T^{(k)}_{f, p}\Big)^{-d}$, while the theorem of A'Campo, \thmref{thm:acampo}, tells us that  $\mathfrak L\big\{T^{(*)}_{f, p}\big\}=0$. Does this tell us anything about the Lefschetz number $\mathfrak L\big\{Q^{(k)}_{f,\nu}\big\}$?

\smallskip

The answer is ``yes'', provided that we assume that the degree $d$ is a prime power.

\medskip

\begin{thm}\label{thm:lef} Suppose that $f$ is homogeneous of degree $\mathfrak p^m$, where $\mathfrak p$ is prime. Suppose that $\nu$ is a complex line in $V(f)$ which contains the origin, and let $p\in \nu-\{\0\}$. Then, in $\Z/\mathfrak p\Z[x]$, there is an equality of characteristic polynomials
$$
\operatorname{char}_{Q^{(k)}_{f,\nu}}(x) = \operatorname{char}_{T^{(k)}_{f,p}}(x) .
$$

In particular, for all $k$, the trace of $Q^{(k)}_{f,\nu}$ is congruent modulo $\mathfrak p$ to the trace of $T^{(k)}_{f, p}$. 
\end{thm}
\begin{proof} We let $A$ denote an integer matrix which represents the action of $\Big(T^{(k)}_{f, p}\Big)^{-1}$ on the free part of $H^k(F_{f, p};\,\Z)$. Then, we believe that it is well-known that, in $(\Z/\mathfrak p\Z)[x]$, the characteristic polynomial $\operatorname{char}_A(x)$ equals $\operatorname{char}_{A^{\mathfrak p^m}}(x)$; however, since the proof is short, we give it.

By the standard argument on  binomial coefficients, we have the following equality modulo $\mathfrak p$:
$$
(xI-A)^{\mathfrak p} = x^{\mathfrak p}I-A^{\mathfrak p}.
$$
Taking determinants, we find that in $\Z/\mathfrak p\Z[x]$, 
$$(\operatorname{char}_A(x))^{\mathfrak p}=\operatorname{char}_{A^{\mathfrak p}}(x^{\mathfrak p}).
$$
By the standard argument on binomial/multinomial coefficients and Fermat's Little Theorem, 
$$(\operatorname{char}_A(x))^{\mathfrak p}=\operatorname{char}_A(x^{\mathfrak p}).
$$
Putting the last two equalities together, we obtain
$\operatorname{char}_A(x^{\mathfrak p})=\operatorname{char}_{A^{\mathfrak p}}(x^{\mathfrak p})$, but this implies that $\operatorname{char}_A(x)=\operatorname{char}_{A^{\mathfrak p}}(x)$.

Iterating this tells us that, in $(\Z/\mathfrak p\Z)[x]$,
$\operatorname{char}_A(x)=\operatorname{char}_{A^{\mathfrak p^m}}(x)$, that is,
$$
\operatorname{char}_{\big(T^{(k)}_{f, p}\big)^{-1}}(x)=\operatorname{char}_{{\big(T^{(k)}_{f, p}\big)^{-\mathfrak p^m}}}(x) = \operatorname{char}_{Q^{(k)}_{f, \nu}}(x),
$$
where the last equality is by \thmref{thm:relmono}. Finally, \corref{cor:inversechar} says that $\operatorname{char}_{\big(T^{(k)}_{f, p}\big)^{-1}}(x)$ is equal to $\operatorname{char}_{T^{(k)}_{f, p}}(x)$, and we are finished.
\end{proof}

\medskip

Combining the above theorem with A'Campo's result, we immediately obtain the following result on the Lefschetz number of $Q^{(*)}_{f, \nu}$:

\begin{cor}\label{cor:lef} Suppose that $f$ is homogeneous of degree $\mathfrak p^m$, where $\mathfrak p$ is prime. Suppose that $\nu\subseteq \Sigma f$ is a complex line which contains the origin. Then, $\mathfrak p$ divides $\mathfrak L\{Q^{(*)}_{f, \nu}\}$. 

\end{cor}

\medskip

Now we can prove:

\begin{thm}\label{thm:mainintro} Suppose that $s:=\dim\Sigma f\geq 1$,  $n-s\geq 1$, and  $f\in \C[z_0, \dots, z_n]$ is homogeneous of degree $\mathfrak p^m$, where $\mathfrak p$ is prime. Let $C$ be an $s$-dimensional components of $\Sigma f$, let $\nu$ be a line through the origin such that $\nu-\{\0\}\subseteq C^\circ$, and let $p\in\nu-\{\0\}$. Let $r_C$ be the remainder when $\mu_C^\circ-(-1)^{n+1-s}$ is divided by $\mathfrak p$.  

Then the algebraic multiplicity $e_\nu(1)$ of the eigenvalue $1$ for $Q^{(n-s)}_{f, \nu}$ acting on $H^{n-s}(F_{f, p};\Z)$ satisfies
$$
e_\nu(1)\leq \mu_C^\circ-\frac{r_C}{2}
$$
or, equivalently
$$
e_\nu(1)\leq\frac{\mu_C^\circ+\mathfrak p\left\lfloor\frac{\mu^\circ_C-(-1)^{n+1-s}}{\mathfrak p}\right\rfloor +(-1)^{n+1-s}}{2},
$$
where $\lfloor\,\rfloor$ denotes the floor function.
\end{thm}
\begin{proof} First note that $n-s\geq 1$ implies that $\widetilde H^{(n-s)}= H^{(n-s)}$. By definition of the remainder, we have non-negative integers $M$ and $r_C$ such that $\mu_C^\circ-(-1)^{n+1-s}=\mathfrak pM+r_C$ where $0\leq r_C<\mathfrak p$. 

Now, the number of eigenvalues of $Q^{(n-s)}_{f, \nu}$, counted with multiplicities is, of course, $\mu_C^\circ$, while $e_\nu(1)$ is the number of these eigenvalues which equal $1$. Denote the trace of $Q^{(n-s)}_{f, \nu}$ by $\operatorname{trc}_\nu$ and consider the trivial equality
$$
\mu_C^\circ=\frac{\mu_C^\circ+\operatorname{trc}_\nu}{2}+\frac{\mu_C^\circ-\operatorname{trc}_\nu}{2}.
$$

We claim that 
\begin{equation}
e_\nu(1)\leq \frac{\mu_C^\circ+\operatorname{trc}_\nu}{2}.\tag{$\dagger$}
\end{equation}
Suppose to the contrary that
$$e_\nu(1)> \frac{\mu_C^\circ+\operatorname{trc}_\nu}{2} = \operatorname{trc}_\nu+ \frac{\mu_C^\circ-\operatorname{trc}_\nu}{2}.
$$
Then there would be less than $\frac{\mu_C^\circ-\operatorname{trc}_\nu}{2}$ eigenvalues which are not $1$ whose sum $S_\nu$ is such that $\operatorname{trc}_\nu= e_\nu(1)+S_\nu$. By \corref{cor:unity}, all of the eigenvalues of $Q^{(n-s)}_{f, p}$ are roots of unity and so the smallest that $S_\nu$ could be would occur if all of the eigenvalues in its sum were $-1$, i.e.,
$$
S_\nu>(-1)\left(\frac{\mu_C^\circ-\operatorname{trc}_\nu}{2}\right).
$$
But then we would have
$$
\operatorname{trc}_\nu= e_\nu(1)+S_\nu> \operatorname{trc}_\nu+ \frac{\mu_C^\circ-\operatorname{trc}_\nu}{2}+(-1)\left(\frac{\mu_C^\circ-\operatorname{trc}_\nu}{2}\right)=\operatorname{trc}_\nu,
$$
a contradiction. Thus, $e_\nu(1)\leq \frac{\mu_C^\circ+\operatorname{trc}_\nu}{2}$.

\medskip

As $n\geq 2$, $F_{f,p}$ is connected and so  the trace of $Q^{(0)}_{f, \nu}$ is equal to $1$. Hence, by \thmref{thm:lef}, $1+(-1)^{n-s}\operatorname{trc}_\nu=\mathfrak p\hat m$ for some integer $\hat m$, or $\operatorname{trc}_\nu=\mathfrak p m+(-1)^{n+1-s}$ where $m=(-1)^{n-s}\hat m$. Of course, since the eigenvalues are all roots of unity, we have the inequality
$$
\mathfrak p m+(-1)^{n+1-s}=\operatorname{trc}_\nu\ \leq \ \mu^\circ_C = \mathfrak p M+r_C +(-1)^{n+1-s}.
$$
Therefore
$$
m\leq M+\frac{r_C}{\mathfrak p},
$$
and, as $m$ and $M$ are integers and $0\leq r_C<\mathfrak p$, we conclude that $m\leq M$ and so
$$
\operatorname{trc}_\nu=\mathfrak p m+(-1)^{n+1-s}\leq\mathfrak p M+(-1)^{n+1-s}=\mu^\circ_C-r_C.
$$
Finally, from ($\dagger$), we find
$$
e_\nu(1)\leq \frac{\mu_C^\circ+\operatorname{trc}_\nu}{2}\leq \frac{\mu_C^\circ+\mu^\circ_C-r_C}{2}=\mu_C^\circ-\frac{r_C}{2},
$$
as desired.

\smallskip

The equivalent inequality in the statement is obtained from this by considering $\mu_C^\circ-(-1)^{n+1-s}=\mathfrak pM+r_C$ and dividing by $\mathfrak p$ to obtain
$$
\frac{\mu_C^\circ-(-1)^{n+1-s}}{\mathfrak p}= M+\frac{r_C}{\mathfrak p},
$$
which implies that 
$$
\left\lfloor\frac{\mu^\circ_C-(-1)^{n+1-s}}{\mathfrak p}\right\rfloor=M.
$$
Thus,
$$
r_C = \mu_C^\circ-(-1)^{n+1-s}-\mathfrak p \left\lfloor\frac{\mu^\circ_C-(-1)^{n+1-s}}{\mathfrak p}\right\rfloor
$$
and the equivalent inequality follows. 
\end{proof}

\medskip

By combining \thmref{thm:supersiersma} with \thmref{thm:mainintro}, we conclude:

\medskip

\begin{thm}\label{thm:mainintro} Suppose that  $f\in \C[z_0, \dots, z_n]$ is homogeneous of degree $\mathfrak p^m$, where $\mathfrak p$ is prime, and let $s:=\dim\Sigma f$. For each $s$-dimensional component $C$ of $\Sigma f$, let $r_C$ be the remainder when $\mu_C^\circ-(-1)^{n+1-s}$ is divided by $\mathfrak p$. 

Then, we have the following upper-bound on the rank of the free abelian group $\widetilde H^{n-s}(F_{f, \0};\Z)$:
$$
\operatorname{rank}\widetilde H^{n-s}(F_{f, \0};\Z)\leq \sum_C \left\lfloor\mu^\circ_C-\frac{r_C}{2}\right\rfloor =\sum_C\left\lfloor\frac{\mu_C^\circ+\mathfrak p\left\lfloor\frac{\mu^\circ_C-(-1)^{n+1-s}}{\mathfrak p}\right\rfloor +(-1)^{n+1-s}}{2}\right\rfloor.
$$
\end{thm}
\begin{proof} This follows immediately from  \thmref{thm:supersiersma} and \thmref{thm:mainintro}, except for the case when $n-s=0$, which once again requires a special small argument.

Suppose $n-s=0$. We use our notation and results from  \propref{prop:special}. So,  consider the irreducible factorization $f=\prod_{k=1}^{r} g_k^{a_k}$ at the origin, and recall that $\mu_k^\circ=a_k-1$ and $\operatorname{rank}\widetilde H^{0}(F_{f, \0};\Z)=\operatorname{gcd}\{a_k\}_k-1$. 

If $\operatorname{gcd}\{a_k\}_k=1$, the desired inequality holds. So suppose that $\operatorname{gcd}\{a_k\}_k>1$. Since  $\operatorname{gcd}\{a_k\}_k$ divides the degree, it must be a positive power of $\mathfrak p$; in particular, $\mathfrak p$ divides each $a_k=\mu_k^\circ+1$. Then, for all $k$, $\mu^\circ_k-(-1)^{n+1-s}=\mu^\circ_k-(-1)^1=a_k$ is divisible by $\mathfrak p$ and so the summands in the desired inequality are simply $\mu_k^\circ$. Hence, the desired inequality is that of Item 2 of  \propref{prop:special}.
\end{proof}

\medskip

\section{Examples}

\begin{exm} Let us first look at an example of Siersma, where it is easy to determine $\operatorname{rank}H^{1}(F_{f, \0};\Z)$ precisely.

\medskip

Let $f=z^2y-xy^2$. So $n=2$, $f$ has degree $3$, and
$$
\Sigma f=V(-y^2, z^2-2xy, 2zy)=V(y,z)= x\textnormal{-axis}.
$$
Thus, there is only one component of $\Sigma f$, $\nu=\C\times\{(0,0)\}$, $s=\dim\Sigma f=1$, and we calculate $\mu_\nu^\circ$ by slicing with $x=1$ and calculating with proper intersections of analytic cycles at $(y,z)=(0,0)$ inside $x=1$:
$$
\mu_\nu^\circ=\left(V(z^2-2y, zy)\right)_\0=\left(V(z^2-2y)\cdot V(zy)\right)_\0 =
$$
$$
\left(V(z^2-2y)\cdot (V(z)+V(y))\right)_\0 = \left(V(-2y,z)\right)_\0+ \left(V(z^2, y)\right)_\0=1+2=3.
$$

\medskip

Thus the ``classical'' bound easy bound from \cite{siersmavarlad} is that rank\,$H^1(F_{f, \0}; \,\Z)\leq 3$, while the bound obtained from \thmref{thm:mainintro} is
$$
\operatorname{rank}H^{1}(F_{f, \0};\Z)\leq  \left\lfloor\frac{3+ 3\left\lfloor\frac{3-(-1)^2}{3}\right\rfloor +(-1)^2}{2}\right\rfloor =2.
$$

\medskip

Note that we obtain the same bound by using Item 5 of \corref{cor:special}.

\medskip

However,  $\operatorname{rank}H^{1}(F_{f, \0};\Z)$ can be easily calculated. As $f$ is homogeneous, $F_{f,\0}$  is diffeomorphic to $f^{-1}(1)$. Thus we consider the set where $z^2y-xy^2=1$. Noting that $y$ cannot be zero, we can solve this for $x$ and obtain
$$
x= \frac{z^2y-1}{y^2}.
$$
Therefore the Milnor fiber is diffeomorphic to $\C^*\times\C$ and so is homotopy-equivalent to a circle. Thus, $\operatorname{rank}H^{1}(F_{f, \0};\Z)=1.$
\end{exm}

\medskip

\begin{exm}

Let $f=u^{25}+w^{24}z-x^{22}yz^2$. So $n=4$ and $\mathfrak p=5$. We find
$$
\Sigma f=V(25u^{24},\, 24w^{23}z,\, -22x^{21}yz^2,\, -x^{22}z^2,\, w^{24}-2x^{22}yz)=
$$
$$
V(u, w, x)\cup V(u, w, z) =C_1\cup C_2.
$$
Thus $s=2$, and so for each $C_i$, $\mu_{C_i}^\circ-(-1)^{n+1-s}= \mu_{C_i}^\circ+1$.

We calculate $\mu^\circ_{C_1}$ and  $\mu^\circ_{C_2}$ by calculating the Milnor numbers at the origin of $f_{y_0, z_0}:=u^{25}+w^{24}z_0-x^{22}y_0z_0^2$ and $f_{x_0, y_0}:=u^{25}+w^{24}z-x_0^{22}y_0z^2$, for generic $x_0$, $y_0$, and $z_0$. We leave it as an exercise to show that  
$$\mu^\circ_{C_1}=24\cdot23\cdot21=11,592\hskip 0.2in\textnormal{and}\hskip 0.2in \mu^\circ_{C_2}=24\cdot23+24\cdot 24=1128.$$

\noindent Therefore the na\"ive bound given by \thmref{thm:supersiersma} is
$$
\operatorname{rank}H^2(F_{f,\0};\,\Z)\leq 11,592+1128=12,720.
$$

\medskip

But now we calculate the $r_{C_i}$'s and apply \thmref{thm:mainintro}. In $\Z/5\Z$, we have
$$
\mu^\circ_{C_1}+1=24\cdot23\cdot21+1=(-1)(-2)(-4)+1=-7=3
$$
and
$$
\mu^\circ_{C_2}+1=24\cdot 47+1=(-1)(2)+1=-1=4.
$$
Thus $r_{C_1}=3$ and $r_{C_2}=4$, and so  \thmref{thm:mainintro} tells us that
\begin{equation}
\operatorname{rank}H^2(F_{f,\0};\,\Z)\leq \left\lfloor 11,592-\frac{3}{2}\right\rfloor + (1128-2)=12, 716,\tag{$\dagger$}
\end{equation}
which, of course, is an improvement.

In this particular example, we can say more. Adding the $u^{25}$ term gives a Milnor fiber which is homotopy-equivalent to the one-point union of $24$ copies of the suspension of the Milnor fiber of $g:=w^{24}z-x^{22}yz^2\in \C[w, x, y, z]$. Hence the rank of $H^2(F_{f,\0};\,\Z)$ must be a multiple of $24$, and so ($\dagger$) implies that
$$
\operatorname{rank}H^2(F_{f,\0};\,\Z)\leq (529)(24)=12, 696.
$$

\medskip

Of course, this leads to the question: if we apply  \thmref{thm:mainintro} to $g=w^{24}z-x^{22}yz^2$, and then multiply by 24, do we obtain the same upper-bound? We leave it as an exercise for the reader to show that \thmref{thm:mainintro} implies that $\operatorname{rank}H^1(F_{g,\0};\,\Z)\leq 528$,
which gives us
$\operatorname{rank}H^2(F_{f,\0};\,\Z)\leq (528)(24)=12, 672$.
\end{exm}

\medskip

\section{Closing remarks}

Suppose $f\in \C[z_0, \dots, z_n]$ is homogeneous of degree $d$, let $s:=\dim\Sigma f$, and assume $s\geq 2$. The knowledgable reader may suspect that L\^e's Attaching Result from \cite{leattach} would help us put a good upper bound on the rank of $\widetilde H^{n-s}(F_{f, \0};\,\Z)$. It is true that iterating Le's result tells us that
$$
\widetilde H^{n-s}(F_{f, \0};\,\Z)\cong \widetilde H^{n-s}(F_{f_{|_M}, \0};\,\Z)\hookrightarrow \widetilde H^{n-s}(F_{f_{|_{M\cap H}}, \0};\,\Z)
$$
where $M$ is a generic linear subspace of codimension $(s-1)$ in $\C^{n+1}$ and $H$ is a generic hyperplane through the origin. Thus $f_{|_M}:\C^{n-s+2}\rightarrow\C$ has a $1$-dimensional critical locus, and $f_{|_{M\cap H}}:\C^{n-s+1}\rightarrow\C$ has a $0$-dimensional critical locus with Milnor number $(d-1)^{n-s+1}$.

So, $(d-1)^{n-s+1}$ is an upper bound on the rank of  $\widetilde H^{n-s}(F_{f, \0};\,\Z)$, but it is far from good. One might hope that analyzing $f_{|_M}:\C^{n-s+2}\rightarrow\C$, with its $1$-dimensional critical locus, would provide a new upper bound on the rank of $\widetilde H^{n-s}(F_{f_{|_M}, \0};\,\Z)$ and thus on the rank of  $\widetilde H^{n-s}(F_{f, \0};\,\Z)$, but we see no way to obtain better bounds via this method than the bounds we obtained in the previous sections.

\bigskip

One might hope to use cohomology with $\Z/\mathfrak p\Z$ coefficients to get a better bound on the rank of  $\widetilde H^{n-s}(F_{f, \0};\,\Z)$ in the case where the degree $(n-s)$ homology (not cohomology) contains $\mathfrak p$-torsion, but then one loses control over the inequalities in \thmref{thm:mainintro}.

\bigskip

Finally, we tried for some time to obtain a result that generalizes \thmref{thm:mainintro} to the case of a homogeneous polynomials of arbitrary degree. We could not find any such useful generalization. Nonetheless, we feel that there must be other number-theoretic restrictions on the cohomology of the Milnor fiber waiting to be found.

\bibliographystyle{plain}

\bibliography{Masseybib}

\end{document}